\documentclass[11pt,a4paper]{article}
\usepackage{mathrsfs}
\usepackage{amsmath,amsthm,color,eepic}
\usepackage{amssymb} \usepackage{verbatim}

\newtheorem{thm}{Theorem}
\newtheorem{cor}[thm]{Corollary}
\newtheorem{lem}[thm]{Lemma}

\newtheorem{conj}[thm]{Conjecture}

\begin{document}
	
	\title{\Large On two cycles of consecutive even lengths}
	\author{
		Jun Gao\thanks{Extremal Combinatorics and Probability Group, Institute for Basic Science (IBS), Daejeon, South Korea. Research supported by the Institute for Basic Science (IBS-R029-C4). Email: gj950211@gmail.com, jungao@ibs.re.kr.}
		\and
		Binlong Li\thanks{School of Mathematics and Statistics, Northwestern Polytechnical University, Xi'an, Shaanxi 710072, P. R. China. Supported by National Natural Science Foundation of China grant 12071370 and Natural Science Foundation of Shaanxi Province (No. 2021JM-040). Email: binlongli@nwpu.edu.cn.}
		\and
		Jie Ma\thanks{School of Mathematical Sciences, University of Science and Technology of China, Hefei 230026, China.
			Research supported by the National Key R and D Program of China 2020YFA0713100, National Natural Science Foundation of China grant 12125106,
			Innovation Program for Quantum Science and Technology 2021ZD0302904, and Anhui Initiative in Quantum Information Technologies grant AHY150200.
			Email: jiema@ustc.edu.cn.}
		\and
		Tianying Xie\thanks{School of Mathematical Sciences, University of Science and Technology of China, Hefei 230026, China.
			Email: xiety@mail.ustc.edu.cn.}
	}
	
	\date{}
	\maketitle
	
	\begin{abstract}
		Bondy and Vince showed that every graph with minimum degree at least three contains two cycles of lengths differing by one or two.
		We prove the following average degree counterpart that
		every $n$-vertex graph $G$ with at least $\frac52(n-1)$ edges, unless $4|(n-1)$ and every block of $G$ is a clique $K_5$, contains two cycles of consecutive even lengths.
		Our proof is mainly based on structural analysis, and a crucial step which may be of independent interest shows that the same conclusion holds for every 3-connected graph with at least 6 vertices.
		This solves a special case of a conjecture of Verstra\"ete.
		The quantitative bound is tight and also provides the optimal extremal number for cycles of length two modulo four.
	\end{abstract}

	\section{Introduction}
	The distribution of cycle lengths in a graph has been widely studied in the literature.
	Er\H{o}ds \cite{Erd76,Erd92,Erd95,Erd97} posted many early problems on this topic and one of such problems (see \cite{BV}) asks whether every
	graph with minimum degree at least three contains two cycles whose lengths differ by one or two.
	This was proved by Bondy and Vince \cite{BV} in the following stronger statement.
	\begin{thm}[Bondy and Vince \cite{BV}]\label{BV:Theorem}
		With the exception of $K_1$ and $K_2$, every graph having at most two vertices of degree less than three contains two cycles of lengths differing by one or two.
	\end{thm}
	There has been extensive research on generalizations (to $k$ cycles of consecutive lengths) and extensions of this result, see \cite{HS98,V00,Fan02,SV08,Ma16,LM18,GM20,GHLM,CY19} (just to mention a few).
	In terms of the minimum degree condition, the following result of \cite{GHLM} provides a good understanding on the distribution of consecutive cycle lengths
	that every graph with minimum degree at least $k+1$ contains $k$ cycles whose lengths form an arithmetic progression with common difference one or two.
	This is tight by considering the clique $K_{k+2}$ and complete bipartite graphs $K_{k+1,n}$.
	It is natural to consider the analogous problem under the average degree condition.
	In this direction, Verstra\"ete \cite{V16} made the following conjecture.
	
	\begin{conj}[Verstra\"ete \cite{V16}; see Conjecture X]\label{Conj-V}
		If $G$ is an $n$-vertex graph not containing $k$ cycles of consecutive even lengths, then $e(G)\leq \frac12(2k+1)(n-1)$, with equality only if every block of $G$ is a clique of order $2k+1$.
	\end{conj}
	
	It is easy to see that if true, the bound is best possible.
	The case $k=1$ is well known (i.e., any $n$-vertex graph with at least $\frac32(n-1)+1$ edges contains an even cycle; see \cite{DLS93}) and the other cases are still open.
	In this paper, we prove the case $k=2$ of this conjecture in the following result.
	\begin{thm}\label{main}
		Let $G$ be a graph of order $n$ with $e(G)\ge \frac{5}{2}(n-1)$. Then $G$ contains two cycles of consecutive even lengths, unless $4|(n-1)$ and every block of $G$ is a clique $K_5$.
	\end{thm}
	This can be viewed as the average degree counterpart of the theorem of Bondy-Vince \cite{BV}.
	We will present the proof of Theorem~\ref{main} in Subsection~\ref{Sec:reduction} by reducing it to two results,
	one of which states that any $3$-connected graph on at least $6$ vertices contains two cycles of consecutive even lengths.
	The proof is mainly based on structural analysis on minimal graphs satisfying the edge-density condition.
	
	Regarding to consecutive cycle lengths, a companion problem is to consider cycle lengths modulo a fixed positive integer $k$.
	This was proposed by Burr and Erd\H{o}s (see \cite{Erd76,Erd95}) in the following question.
	Let $k> \ell \geq 0$ be integers such that the congruence class $\ell \pmod k$ contains some even integers.
	Is it true that there exists a least constant $c_{\ell,k}$ such that every $n$-vertex graph with at least $c_{\ell,k}\cdot n$ edges contains a cycle of length $\ell$ modulo $k$?
	This was proved affirmatively by Bollob\'as \cite{B77},
	and there has been a substantial body of subsequent research studying this topic, see \cite{Th83,Th88,Dean,S92,DLS93,CS,V00,Fan02,D10,Ma16,SV17,LM18,LM19,GHLM,CY19,MS22}.
	Erd\H{o}s \cite{Erd95} asked for exact values of $c_{\ell,k}$,
	however to date the only known case is the folklore result $c_{0,2}=3/2$.
	Using Theorem~\ref{main}, one can get the following new case easily.
	
	\begin{cor}\label{cor}
		It holds that $c_{2,4}=5/2$.
	\end{cor}

	\subsection{Proof by reduction}\label{Sec:reduction}
	We now present the proof of Theorem~\ref{main} by reducing to the following two results.
	\begin{thm}\label{admissible: k=3}
		Let $G$ be a graph with $x, y\in V(G)$ such that $G+xy$ is 2-connected. If every vertex of $G$
		other than $x$ and $y$ has degree at least $3$, and any edge $uv\in E(G)$ with $\{u,v\}\cap \{x,y\} = \emptyset$ has degree sum $d_G(u)+d_G(v)\ge 7$,
		then there exist two paths from $x$ to $y$ in $G-xy$ whose lengths differ by two.
	\end{thm}
	
	\begin{thm}\label{3-connected}
		Every 3-connected graph of order at least 6 contains two cycles of consecutive even lengths.
	\end{thm}

	\begin{proof}[\bf Proof of Theorem~\ref{main}]
		Let $G$ be a minimum counterexample of Theorem~\ref{main}.
		That is, $G$ is a graph of order $n$ with $e(G)\ge \frac{5(n-1)}{2}$ which does not satisfy Theorem~\ref{main}, but any graph $H$ with $|V(H)|<n$ and $e(H)\ge \frac{5(|V(H)|-1)}{2}$ holds for Theorem~\ref{main}.
		It is easy to see that $G$ is connected with minimum degree $\delta(G)\ge 3$.
		One can also show that any edge $uv\in E(G)$ satisfies $d_G(u)+d_G(v)\ge 7$.
		To see this, say there exists an edge $uv$ with $d_G(u)+d_G(v)\leq 6$. Then the subgraph $G-\{u,v\}$ has $n-2$ vertices and at least $5(n-3)/2$ edges;
		by the minimality choice of $G$, either $G-\{u,v\}$ (and thus $G$) contains two cycles of consecutive even lengths, or $4|(n-3)$ and every block of $G-\{u,v\}$ is a clique $K_5$.
		In the latter case, we can infer that $d_G(u)+d_G(v)=6$ and $G$ is connected with every block being $K_5$, so when we put the edge $uv$ back,
		one can easily find two cycles of consecutive even lengths in $G$, a contradiction.
		
		We now claim that $G$ is 2-connected. Suppose not. Then there exists a cut-vertex $v$ of $G$.
		Let $B$ be a component of $G-v$. Let $G_2 = G- B$ and $G_1 = G-V(G_2-v)$.
		We must have $e(G_1)\le \frac{5(|V(G_1)|-1)}{2}$ and $e(G_2)\le \frac{5(|V(G_2)|-1)}{2}$,
		with equality if and only if every block of $G_1$ and $G_2$ is a clique $K_5$ (otherwise, by the minimum assumption on $G$, some $G_i$ and thus $G$ contains two cycles of consecutive  even lengths, a contradiction).
		Then $\frac{5(n-1)}{2}\le e(G) =e(G_1) + e(G_2) \le \frac{5(n-1)}{2}$, so the equality holds, which easily implies that every block of $G$ is a clique $K_5$, a contradiction.
		
		By Theorem~\ref{3-connected}, we have $G$ is not 3-connected.
		So there exists a $2$-cut $\{x,y\}$ such that $G-\{x,y\}$ has at least two components, say $F_1,F_2$.
		Let $H_i$ be the graph induced by $V(F_i)\cup \{x,y\}$ for $i\in \{1,2\}$.
		Since $G$ is 2-connected, each $H_i+xy$ is 2-connected and satisfies the conditions of Theorem~\ref{admissible: k=3}.
		By Theorem~\ref{admissible: k=3}, there exist two paths from $x$ to $y$ in $H_1-xy$ whose lengths differ by two.
		If $H_2-xy$ is not bipartite, then there exist an odd path and an even path between $x$ and $y$ in $H_2-xy$.
		So $G$ contains two cycles of consecutive even lengths, a contradiction.
		Now we assume $H_2-xy$ is bipartite. Then by Theorem~\ref{BV:Theorem}, $H_2-xy$ (and thus $G$) contains two cycles of consecutive even lengths, finishing the proof.
	\end{proof}

	\subsection{Notation and organization}
	
	All graphs in this paper are finite, undirected, and simple.
	Let $G$ be a graph and $H$ be its subgraph. 
	Define $N_G(H):=\left(\bigcup_{v\in V(H)}N_G(v)\right)\setminus V(H)$ and $N_G[H]:=N_G(H)\cup V(H)$.
	Let $S$ be a subset of $V(G)$.
	We define $G[S]$ to be the subgraph induced by $S$ in $G$, and $G-S$ to be the subgraph $G[V(G)\setminus S]$.
	We say that a graph $G'$ is obtained from $G$ by {\it contracting $S$} into a vertex $s$,
	if $V(G')=(V(G)\setminus S)\cup\{s\}$ and $E(G')=E(G-S)\cup\{vs: v\in V(G)\setminus S$ is adjacent to some vertex of $S$ in $G\}$.
	For two distinct vertices $x,y$ of $G$, let $G+xy$ denote the graph obtained from $G$ by adding a new edge $xy$,
	and let $G-xy$ denote the graph obtained from $G$ by removing the edge $xy$.
	A vertex is a {\it leaf} of $G$ if it has degree one in $G$.
	We say a vertex subset $A$ of $G$ is a {\em cut} of $G$ if $G-A$ contains more components than $G$.
	A vertex is called a {\it cut-vertex} of $G$ if it forms a cut of $G$.
	A {\em block} $B$ of $G$ is a maximal connected subgraph of $G$ such that $B$ contains no cut-vertex.
	So, a block is either an isolated vertex, an edge or a $2$-connected graph.
	A block in $G$ is an {\em end-block} if and only it containing at most one cut-vertex of $G$.
	
	A {\em theta-graph} is a graph consisting of three internally vertex-disjoint paths between two fixed vertices.
	For a cycle $C$, we use $\overrightarrow{C}$ to express this cycle with a prescribed orientation,
	and for vertices $u,v\in V(C)$, the notation $\overrightarrow{C}[u,v]$ denotes the subpath of $C$ from $u$ to $v$ following the given orientation.
	If $P$ is a walk, a path or a cycle, then we use $\ell(P)$ to denote the {\it length} of $P$, i.e., the number of edges of $P$.
	Two edges are called {\it independent} if they share no common vertex.
	Throughout the rest of the paper, we will reserve the word {\it disjoint} for vertex-disjoint.
	
	The rest of the paper is organized as follows.
	In Section~\ref{Graph with two cut}, we prove Theorem~\ref{admissible: k=3}.
	In Section~\ref{sec:3-connected}, we prove Theorem~\ref{3-connected}.
	In Section~\ref{sec:conclusion}, we present a generalization of Theorem~\ref{admissible: k=3} and discuss several related open problems.
	
	\section{Proof of Theorem \ref{admissible: k=3}}\label{Graph with two cut}
	Let $G$ be a graph with two vertices $x, y\in V(G)$ such that $G+xy$ is 2-connected,
	every vertex of $G$ other than $x$ and $y$ has degree at least $3$, and any edge $uv\in E(G)$ with $\{u,v\}\cap \{x,y\} = \emptyset$ has degree sum $d_G(u)+d_G(v)\ge 7$.
	We aim to show that there exist two paths from $x$ to $y$ in $G-xy$ whose lengths differ by two.
	Without loss of generality, we may assume that $d_G(y) \ge d_G(x)$ and $xy\notin E(G)$.
	
	We prove by induction on $|V(G)|$. By the condition we have $|V(G)|\ge 5$.
	It is easy to check that Theorem \ref{admissible: k=3} holds for the base case $|V(G)|=5$.
	Now let us assume that Theorem \ref{admissible: k=3} holds for any graph $H$ with $|V(H)|<|V(G)|$.
	We will divide the rest of the proof into two cases.
	
	\medskip
	
	\textbf{Case 1.} There exists a cycle of length four containing $x$ in $G-y$.
	
	\medskip
	
	Let $C = xx_1ax_2x$ be a cycle of length four in $G-y$.
	Let $F$ be the component of $G-V(C)$ containing $y$.
	If $F$ is adjacent to $x_1$, let $P$ be a path from $x_1$ to $y$ in $G$.
	Then $xx_1\cup P$ and $xx_2ax_1\cup P$ are two paths from $x$ to $y$ whose lengths differ by two.
	So $F$ is not adjacent to $x_1$. Similarly we know that $F$ is not adjacent to $x_2$.
	Let $G' = G-V(F)$.
	Then $G'$ together with vertices $x,a$ satisfy the induction hypothesis.
	So there exist two paths from $x$ to $a$ in $G'$ whose lengths differ by two.
	Since $G+xy$ is 2-connected (note that $F$ is not adjacent to $x_1$ and $x_2$),
	$F$ must be adjacent to $a$. Let $P'$ be a path form $a$ to $y$ with all internal vertices in $F$.
	Concatenating those paths, we can get two paths from $x$ to $y$ in $G$ whose lengths differ by two, as desired.

	\medskip
	
	\textbf{Case 2.}  There does not exist a cycle of length four containing $x$ in $G-y$.
	
	\medskip
	
	Let $X=N_G(x)$. Note that $y\notin X$. Let $G^*$ be the graph obtained from $G-x$ by contracting $X$ into a new vertex $x^*$.
	Since $G-y$ does not have a cycle of length four containing $x$, we know that any vertex $v$ other than $x^*$ and $y$ has degree $d_{G^*}(v) = d_G(v)$.
	If $G^*+x^*y$ is not 2-connected, then $x^*$ is the only cut-vertex of $G^*$ and we let $B$ be the block of $G^*$ containing $y$.
	If $G^*+x^*y$  is 2-connected, we let $B=G^*+x^*y$.

	If $|V(B)|\ge 3$, then $B$ is 2-connected.
	In this case, $B$ together with vertices $x^*,y$ satisfy the induction hypothesis.
	So there exist two paths $P_1$ and $P_2$ from $x^*$ to $y$ in $B$ whose lengths differ by two.
	Putting back to $G$, we see that there exist $x_1,x_2\in X$ (which may be the same) and paths $P'_i$ from $x_i$ to $y$ for $i\in \{1,2\}$ in $G$ with $\ell(P'_2) =\ell(P'_1)+2$.
	Then $xx_1\cup P'_1$ and $xx_2\cup P'_2$ are two paths from $x$ to $y$ in $G$ of lengths differing by two.
	
	It remains to consider the case when $|V(B)|=2$, i.e., $B = x^*y$.
	This gives that $N_G(y)\subseteq X$.
	Since $d_G(y)\ge d_G(x)$ and $xy\notin E(G)$, we have $N_G(y)=X=N_G(x)$.
	Since $G-y$ does not have a cycle of length four containing $x$, any vertex in $X$ have at most one neighbor in $X$.
	Since every vertex of $G$ other than $x$ and $y$ has degree at least $3$ and any edge $uv\in E(G)$ with $\{u,v\}\cap \{x,y\}= \emptyset$ has $d_G(u)+d_G(v)\ge 7$,
	we can see that $G\setminus (X\cup\{x,y\}) \ne \emptyset$, which implies $G^*\ne B$.
	
	Hence, there exists a block $D'$ of $G^*$ other than $B$.
	Let $D$ be obtained from $D'$ by removing $x^*$.
	Since $G +xy$ is 2-connected, we have $|X|\ge 2$ and $|N_G(D)\cap X|\ge 2$.
	Let $u_1$ be a vertex in $N_G(D)\cap X$ and $G_1$ be the graph obtained from $G[X\cup D]$ by contracting $X\setminus \{u_1\}$ into a new vertex $u_2$.
	For any vertex $v\in D$, we have $d_{G_1}(v)=d_G(v)$.
	Since $D'$ is 2-connected, we see $G_1+u_1u_2$ is 2-connected.
	Then $G_1+u_1u_2$ together with vertices $u_1, u_2$ satisfy the induction hypothesis.
	So there exist two paths $P_1$ and $P_2$ from $u_1$ to $u_2$ in $G_1$ of lengths differing by two.
	Back to $G$, we see there exist vertices $x_1, x_2\in X\setminus\{u_1\}$ (which may be the same) and paths $P'_i$ from $x_i$ to $u_1$ for $i\in \{1,2\}$ in $G$ with $\ell(P'_2) =\ell(P'_1)+2$.
	Then $xx_1\cup P'_1\cup u_1y$ and $xx_2\cup P'_2\cup u_1y$ are two paths from $x$ to $y$ in $G$ of lengths differing by two.
	This completes the proof of Theorem \ref{admissible: k=3}. \qed

	\section{Proof of Theorem~\ref{3-connected}}\label{sec:3-connected}
	
	This section will be devoted to the proof of Theorem~\ref{3-connected}, stating that every 3-connected graph $G$ with at least 6 vertices contains two cycles of consecutive even lengths.
	We will start with a series of lemmas.

	\begin{lem}\label{non-seperated even cycle}
		Let $G$ be a 3-connected graph and let $D$ be a fixed connected subgraph of $G$ such that $G-V(D)$ contains an even cycle.
		If $C$ is an even cycle of $G-V(D)$ such that the component of $G-V(C)$ containing $D$ is maximal,
		then $G-V(C)$ is connected. Moreover, such $C$ has at most one chord, and if $u_0v_0$ is a chord of $C$,
		then both paths $\overrightarrow{C}[u_0,v_0]$ and $\overrightarrow{C}[v_0,u_0]$ have even lengths.
	\end{lem}
	
	\begin{proof}
		Let $C$ be an even cycle of $G-V(D)$ such that the component (say $F$) of $G-V(C)$ containing $D$ is maximal.
		We first show that $G-V(C)$ is connected.
		Suppose not. Then there is another component $H$ of $G-V(C)$.
		Since $G$ is 3-connected, there exists at least three neighbors of $H$ on $C$, say $v_0,v_1,v_2$,
		such that there are three internal disjoint paths $P_0,P_1,P_2$ from a vertex $u\in V(H)$ to $v_0,v_1,v_2$,
		respectively, with all internal vertices in $H$.
		Let $\overrightarrow{C}$ denote the orientation of $C$ such that $v_0,v_1,v_2$ appear in $C$ in this cyclic order.
		Suppose that $F$ has a neighbor in $V(C)\backslash\{v_0,v_1,v_2\}$, say $v\in V(\overrightarrow{C}[v_2,v_0])\setminus \{v_2,v_0\}$.
		Note that the theta-graph $P_0\cup P_1\cup P_2\cup \overrightarrow{C}[v_0,v_2]$ contains an even cycle $C'$.
		The component of $G-V(C')$ containing $D$ clearly contains $F\cup \{v\}$, contradicting the property of $C$.
		Since $G$ is 3-connected, we can conclude that $N_C(F)=\{v_0,v_1,v_2\}=N_C(H)$.
		For $i\in \{0,1,2\}$, let $C_i=\overrightarrow{C}[v_i,v_{i+1}]\cup P_i\cup P_{i+1}$, where the subscript is modulo $3$.
		Then we have $$\sum_{0\leq i\leq 2}\ell(C_i)=\ell(C)+2\sum_{0\leq i\leq 2}\ell(P_i)$$ is even.
		It follows that at least one of the three cycles (say $C_i$) is even.
		Then the component of $G-V(C_i)$ containing $F$ also contains the vertex $v_{i+2}$, again contradicting the property of $C$.
		This finishes the proof that $G-V(C)$ is connected.
		
		If there is an even cycle $C'$ with $V(C')\subsetneq V(C)$,
		then the component $F'$ of $G-V(C')$ containing $D$ also contains $F$.
		As $F$ is maximal (subject to the choice of $C$),
		we conclude that $F'=F$ and thus $C'$ is also an even cycle $G-V(D)$ such that the component of $G-V(C')$ containing $D$ is maximal.
		By the above proof, $G-V(C')$ is connected, implying that $|V(F')|>|V(F)|$, a contradiction.
		So there exists no even cycle $C'$ with $V(C')\subsetneq V(C)$.
		It follows that $C$ has at most one chord, and if $C$ has a chord $u_0v_0$,
		then both paths $\overrightarrow{C}[u_0,v_0]$ and $\overrightarrow{C}[v_0,u_0]$ have even lengths.
	\end{proof}
	
	For an even cycle $C$, two vertices $u,v\in V(C)$ are \emph{quasi-diagonal} in $C$ if $\ell(\overrightarrow{C}[u,v])=\ell(C)/2-1$ or $\ell(C)/2+1$.
	Note that every vertex in $C$ has exactly two quasi-diagonal vertices.
	
	\begin{lem}\label{two:cycle}
		Let $G$ be a graph. Let $B$ be an even cycle and $D$ be an odd cycle in $G$.
		If one of the following two situations is met:
		\begin{itemize}
			\item [(1).] $B$ and $D$ are disjoint and there exist two disjoint paths $P_1, P_2$ between $B$ and $D$ such that their endpoints in $B$ are quasi-diagonal, or
			\item [(2).] $B$ and $D$ have exactly one common vertex $u$ and there exists a path between $B-u$ and $D-u$ such that its origin in $B$ are quasi-diagonal with $u$,
		\end{itemize}
		then there are two cycles of consecutive even lengths in $G$.
	\end{lem}
	
	\begin{proof}
		First let us show the proof for (1). For $i\in \{1,2\}$, let $u_i$ be the endpoint of $P_i$ in $B$ and $v_i$ be its endpoint in $D$.
		As $D$ is odd, the two paths $Q_1:=P_1\cup P_2\cup\overrightarrow{D}[v_1,v_2]$ and $Q_2:=P_1\cup P_2\cup\overrightarrow{D}[v_2,v_1]$ have lengths with different parity.
		Then one of these paths, say $Q_j$, satisfies $\ell(Q_j)\equiv \ell(B)/2-1(\bmod ~2)$.
		Since $u_1, u_2$ are quasi-diagonal in $B$, the two cycles $Q_j\cup\overrightarrow{B}[u_1,u_2]$ and $Q_j\cup\overrightarrow{B}[u_2,u_1]$ are two cycles of consecutive even lengths.
		The proof for (2) is identical to (1).
	\end{proof}
	
	In the following three lemmas, we show that two cycles of consecutive even lengths can be constructed in 3-connected graphs
	using two disjoint (or almost disjoint) cycles of prescribed parities.
	
	\begin{lem}\label{ two vertex-disjoint cycles: odd even}
		Let $G$ be a 3-connected graph. If there are two disjoint cycles one of which is even and the other is odd, then $G$ contains two cycles of consecutive even lengths.
	\end{lem}
	
	\begin{proof}
		By the condition, there exists an odd cycle $D$ in $G$ such that $G-V(D)$ contains an even cycle.
		Fix such $D$. Let $C$ be an even cycle in $G-V(D)$ such that the component $F$ of $G-V(C)$ containing $D$ is maximal.
		Using Lemma~\ref{non-seperated even cycle}, we see that $F=G-V(C)$ is connected.
		Moreover, $C$ has at most one chord, and if $u_0v_0$ is a chord of $C$,
		then both paths $\overrightarrow{C}[u_0,v_0]$ and $\overrightarrow{C}[v_0,u_0]$ have even lengths.
		
		By the 3-connectedness of $G$, there are three disjoint paths $P_1,P_2,P_3$ between $C$ and $D$ in $G$.
		If $C$ is a $C_4$, then there exist two paths $P_i,P_j$ such that their endpoints in $C$ are adjacent, which are quasi-diagonal.
		Thus by Lemma~\ref{two:cycle}, $G$ has two cycles of consecutive even lengths.
		From now on, we may assume that $\ell(C)\geq 6$.

		We construct an auxiliary graph $Qdi(C)$ on $V(C)$ such that two vertices are adjacent in $Qdi(C)$ if and only if they are quasi-diagonal in $C$.
		Note that if $\ell(C)=0\ (\bmod\ 4)$, then $Qdi(C)$ is a cycle, and if $\ell(C)=2\ (\bmod\ 4)$, then $Qdi(C)$ is a disjoint union of two odd cycles.
		We consider two cases as follows.
		
		\medskip
		
		\textbf{Case 1.} $\ell(C)=2 \ (\bmod ~4)$.
		
		\medskip
		Write $C=x_1x_2...x_{4t+2}x_1$ for some $t\geq 1$.
		Then in this case, $Qdi(C)$ is a disjoint union of two odd cycles (say $C_1, C_2$) of equal length such that $V(C_1)=\{x_1,x_3,...,x_{4t+1}\}$ and $V(C_2)=\{x_2,x_4,...,x_{4t+2}\}$.
		If $C$ has a chord $u_0v_0$, then as pointed out above, both paths $\overrightarrow{C}[u_0,v_0]$ and $\overrightarrow{C}[v_0,u_0]$ have even lengths,
		implying that both $u_0,v_0$ are contained in either $V(C_1)$ or $V(C_2)$.
		Let $Q$ be the cycle of $Qdi(C)$ not containing $u_0,v_0$.
		If $C$ has no chord, then let $Q$ be any cycle of $Qdi(C)$.
		
		Let $Q=u_1u_2\ldots u_k$. So $k=\ell(Q)=\ell(C)/2$ is odd.
		As $C$ has at most one chord, by the choice of $Q$, every vertex $u_i$ has exactly two neighbors in $V(C)$ and thus has a neighbor $v_i\in V(F)$ in the graph $G$.
		Let $T_0$ be a spanning tree of $F$ and $T$ be the tree obtained from $T_0$ by adding the $k$ edges $u_iv_i$.
		Let $P_i$ and $P'_i$ be the paths of $C$ between $u_i$ and $u_{i+1}$ of length $\ell(C)/2-1$ and $\ell(C)/2+1$, respectively,
		and let $Q_i$ be the unique path of $T$ between $u_i$ and $u_{i+1}$. Thus $C_i=P_i\cup Q_i$ and $C'_i=P'_i\cup Q_i$ are two cycles of lengths differing by two.
		
		Let $W=u_1Q_1u_2Q_2\ldots u_kQ_ku_1$ be a closed walk in $T$.
		If we view $T$ as a rooted tree with any fixed root $r$, then the parity of each $\ell(Q_i)$ agrees with the parity of $d_T(u_i,r)-d_T(u_{i+1},r)$, where $d_T(u_i,r)$ denotes the length of the unique path of $T$ between $u_i$ and $r$.
		This shows that $\ell(W)$ must be even.
		Clearly each $P_i$ have length $k-1$ and every edge of $C$ is contained in $k-1$ many paths $P_i$. Thus
		$$\sum_{i=1}^k\ell(C_i)=\sum_{i=1}^k\ell(P_i)+\sum_{i=1}^k\ell(Q_i)=(k-1)\ell(C)+\ell(W),$$
		which is even (as $k$ is odd and $\ell(W)$ is even). This implies that at least one cycle $C_i$ is even.
		Thus $C_i$ and $C'_i$ are two cycles of consecutive even lengths in $G$, as desired.
		
		\medskip
		
		\textbf{Case 2.}  $\ell(C)=0 \ (\bmod ~4)$.
		
		\medskip
		
		If $F=G-V(C)$ is 2-connected, let $B=F$; otherwise, let $B$ be the block of $F$ containing the odd cycle $D$.
		For any vertex $v\in V(B)$, let $H_v$ be the induced subgraph of $F$ with vertex set
		$$V(H_v)=\{u\in V(F): \mbox{there is a path between }u \mbox { and } v\mbox{ with all edges in }E(F)\backslash E(B)\}.$$
		So $H_v$ be the component of $F-E(B)$ containing $v$, and possibly $V(H_v)=\{v\}$.
		(In fact $|V(H_v)|\geq 2$ if and only if $v$ is a cut-vertex of $F$.)
		We call such graphs \emph{$B$-branches}.
		We claim that for any two vertices $u_1,u_2\in V(F)$, there are two disjoint paths between $\{u_1,u_2\}$ and $V(D)$ in $F$ if and only if $u_1,u_2$ are in distinct $B$-branches.
		If $u_1,u_2$ are in a same $B$-branch $H_v$, then each path between $\{u_1,u_2\}$ and $V(D)$ passes through $v$.
		Suppose now $u_1\in V(H_{v_1}),u_2\in V(H_{v_2})$ with $v_1\neq v_2$.
		Let $P_i$ be a path of $H_{v_i}$ between $u_i$ and $v_i$.
		Since $B$ is 2-connected, there are two disjoint paths $Q_1,Q_2$ from $v_1,v_2$, respectively to $D$.
		Then $P_1\cup Q_1$ and $P_2\cup Q_2$ are two disjoint paths between $\{u_1,u_2\}$ and $V(D)$ in $F$.
		
		Let $u_1,u_2$ be any two quasi-diagonal vertices in $C$.
		If each $u_i$ has a neighbor $w_i$ for $i\in \{1,2\}$ such that $w_1, w_2$ belong to distinct $B$-branches,
		then there are two disjoint paths between $C$ and $D$ such that their endpoints in $C$ are quasi-diagonal.
		By Lemma~\ref{two:cycle}, $G$ has two cycles of consecutive even lengths.
		So we conclude that for any two quasi-diagonal vertices $u_1,u_2$ in $C$ (i.e., an edge $u_1u_2$ in $Qdi(C)$),
		their neighbors in $F$ all belong to a common $B$-branch.	
		Recall that $Qdi(C)$ is a cycle here.
		If $C$ has no chord, then every vertex of $C$ has at least one neighbor in $F$ and thus all neighbors of $C$ are contained in a common $B$-branch say $H_v$.
		Then $v$ is a cut-vertex of $G$, a contradiction.
		It remains to consider when $C$ has a unique chord $u_0v_0$.
		Then $Qdi(C)-\{u_0,v_0\}$ has at most two components, say $L_1, L_2$.
		For each $i\in \{1,2\}$, $N_G(V(L_i))\cap V(F)$ is contained in a common $B$-branch, say $H_{v_i}$.
		If $v_1\neq v_2$, then by the above proof, $N_G(\{u_0,v_0\})\cap V(F)=\emptyset$; otherwise $v_1=v_2$, then $N_G(\{u_0,v_0\})\cap V(F)$ is contained in $H_{v_1}$ as well.
		Putting all together, we see that $\{v_1,v_2\}$ is a cut of $G$ with size at most two, a contradiction to the 3-connectedness of $G$.
		This proves Lemma~\ref{ two vertex-disjoint cycles: odd even}.
	\end{proof}
	
	\begin{lem}\label{ two vertex-disjoint cycles: odd even u}
		Let $G$ be a 3-connected graph.
		If $G$ has an even cycle and an odd cycle which share exactly one common vertex,
		then $G$ contains two cycles of consecutive even lengths.
	\end{lem}
	
	\begin{proof}
		Let $B$ be an even cycle and $D$ an odd cycle in $G$ such that $V(B)\cap V(D)=\{u\}$.
		By Lemma~\ref{non-seperated even cycle}, there is an even cycle $C$ of $G-V(D-u)$ such that $F=G-V(C)$ is connected.
		If $u\notin V(C)$, then $C$ and $D$ are two disjoint cycles with opposite parities
		and thus by Lemma~\ref{ two vertex-disjoint cycles: odd even}, $G$ contains two cycles of consecutive even lengths.
		Thus we may assume that $u\in V(C)$.
		
		Let $u_1,u_2$ be the two vertices in $C$ each of which is quasi-diagonal with $u$.
		If $u_i$ has a neighbor $v$ in $F$, then there exists a path $P$ in $F$ from $v$ to $V(D-u)$,
		implying that $u_iv\cup P$ is a path between $C$ and $D$ whose endpoint in $C$ is quasi-diagonal with $u$.
		By Lemma~\ref{two:cycle}, $G$ has two cycles of consecutive even lengths.
		So $u_1, u_2$ have no neighbors in $F$.
		Then $u_1u_2$ must be a chord of $C$ (i.e., the only chord of $C$).
		We also have $\ell(C)\geq 6$; as otherwise $C$ is a $C_4$ say $C=uu_1wu_2u$, then $\{u,w\}$ is a 2-cut of $G$, a contradiction to the 3-connectedness of $G$.
		Let $u_3$ be a neighbor of $u$ on $C$ and $u_4$ be the common neighbor of $u_1,u_2$ on $C$.
		As $u_3$ has a neighbor (say $v$) in $F$, there exists a path $Q$ in $F$ from $v$ to $V(D-u)$.
		Now $uu_3v\cup Q\cup D$ forms a theta-graph and thus contains an even cycle $C'$.
		Clearly $C'$ is disjoint with the triangle $u_1u_2u_4u_1$.
		By Lemma~\ref{ two vertex-disjoint cycles: odd even}, $G$ contains two cycles of consecutive even lengths.
	\end{proof}
	
	\begin{lem}\label{two odd cycles}
		Let $G$ be a 3-connected graph. If there are two disjoint odd cycles, then $G$ contains two cycles of consecutive even lengths.
	\end{lem}
	
	\begin{proof}
		We choose two disjoint odd cycles $B,D$ in $G$ such that $|V(B)|$ is minimum.
		Let $F=G-V(B)$. If $F$ has an even cycle, then $G$ contains two cycles of consecutive even lengths by Lemma~\ref{ two vertex-disjoint cycles: odd even}.
		So $F$ has no even cycle. This also says that $F$ contains no theta-graph,
		hence every block of $F$ is either $K_1$, $K_2$ or an odd cycle. In particular, $D$ is a block of $F$.
		
		We claim that $F=D$. Suppose not, then there exists an end-block $D'$ of $F$ other than $D$.
		If $D'=K_1$, let $v$ be this isolated vertex.
		If $D'=K_2$, let $v\in V(D')$ be a non-cut-vertex of $F$.
		Since $G$ is 3-connected, $v$ has at least two neighbors $u_1,u_2$ in $B$.
		Then $B\cup\{u_1v,u_2v\}$ is a theta-graph and thus contains an even cycle disjoint with the odd cycle $D$.
		By Lemma~\ref{ two vertex-disjoint cycles: odd even}, $G$ contains two cycles of consecutive even lengths.	
		Now we consider when $D'$ is an odd cycle. Denote the unique cut-vertex of $F$ contained in $D'$ by $v$ (if it exists).
		Since $G$ is 3-connected, there are two independent edges $u_1v_1,u_2v_2$ between $B$ and $D'-v$.
		Let $P$ be the path of $D'-v$ between $v_1$ and $v_2$.
		Then $B\cup P\cup\{u_1v_2,u_2v_2\}$ is a theta-graph which is disjoint with $D$,
		and by Lemma~\ref{ two vertex-disjoint cycles: odd even} again, $G$ contains two cycles of consecutive even lengths.
		
		Note that $V(G)=V(B)\cup V(D)$ and both $B$ and $D$ are induced odd cycles in $G$.
		If there is a vertex $u\in V(B)$ having at least two neighbors in $D$,
		then $G[V(D)\cup\{u\}]$ has a theta-graph and thus contains an even cycle which shares exactly one common vertex with $B$.
		By Lemma~\ref{ two vertex-disjoint cycles: odd even u}, $G$ contains two cycles of consecutive even lengths.
		Hence we may assume that every $x\in V(B)$ has exactly one neighbor, denoted by $x'$, in $D$.
		By the same analysis we see that every $y\in V(D)$ has exactly one neighbor, denoted by $y'$, in $B$.
		It follows that $|V(B)|=|V(D)|$ and $G$ is a cubic graph with $|V(G)|=2\ (\bmod\ 4)$.

		For any edge $uv\in E(B)$, let $D_o(uv)$ be the unique odd path in $D$ with endpoints $u',v'$ and $D_e(uv)$ be the unique even path in $D$ with the endpoints $u',v'$.
		Similarly we can define $B_o(uv)$ and $B_e(uv)$ for every $uv\in E(D)$.
		If every edge $uv\in E(B)$ has $\ell(D_o(uv))=1$, then it is easy to find $C_4$ and $C_6$ in $G$, as desired.
		If there exists some $uv\in E(B)$ with $\ell(D_o(uv))\geq 5$,
		then $D^*:=D_e(uv) \cup \{uu',uv,vv'\}$ is an odd cycle;
		note that $D_o(uv)\setminus \{u',v'\}$ has at least three vertices each of which has a neighbor in $V(B)\setminus \{u,v\}$,
		thus the induced subgraph of $G$ on $\big(V(D_o(uv))\cup V(B)\big)\setminus \{u',v',u,v\}$ contains a theta-graph,
		which contains an even cycle that is disjoint with the odd cycle $D^*$;
		so by Lemma~\ref{ two vertex-disjoint cycles: odd even}, $G$ contains two cycles of consecutive even lengths.
		Therefore we may assume that there exists $uv\in E(B)$ with $\ell(D_o(uv))= 3$.
		Note that $D_o(uv) \cup \{uu',uv,vv'\}$ is a $C_6$ in $G$.	
		If there is some $f\in E(B)$ with $\ell(D_o(f))=1$, then $G$ contains $C_4$ and $C_6$.
		Hence from now on, we may assume that any edge $f\in E(B)$ has $\ell(D_o(f))=3$.
		At this point let us notice that as $B$ and $D$ are symmetric, all above assertions hold for edges in $D$ as well.
		Write $D_o(uv) = u'abv'$.
		Recall that $a',b'$ are the vertices in $B$ adjacent to $a,b$, respectively.
		Similarly we have $\ell(B_o(u'a))=\ell(B_o(ab))=\ell (B_o(v'b))=3$.
		If $B_o(u'a)$ does not contain $v$, then $B_o(u'a)\cup \{a'a,ab,bv',v'v,vu\}$ is a $C_8$ and so $G$ contains $C_6$ and $C_8$.
		So $B_o(u'a)$ contains $v$.
		Similarly, we infer that $B_o(bv')$ contains $u$.
		But this gives $\ell(B_o(ab))=5$, a contradiction.
	\end{proof}
	
	Finally we are ready to present the proof of Theorem~\ref{3-connected}.
	
	\begin{proof}[\bf Proof of Theorem~\ref{3-connected}]
		Let $G$ be a 3-connected graph with at least 6 vertices.
		Clearly, $\delta(G)\ge 3$.
		If $G$ is bipartite, then by Theorem~\ref{BV:Theorem}, $G$ contains two cycles of consecutive even lengths.
		So we assume that $G$ is non-bipartite.
		Let $D$ be a shortest odd cycle of $G$. So $D$ has no chord.
		
		If $G-V(D)$ contains a cycle, then by Lemma~\ref{ two vertex-disjoint cycles: odd even} or Lemma~\ref{two odd cycles},
		$G$ contains two cycles of consecutive even lengths.
		So we may assume that $G-V(D)$ is a forest. Let $F$ be a component (i.e., a tree) of $G-V(D)$ with maximum number of vertices.	
		If $D$ has a vertex $v$ that has three neighbors in $F$,
		then $G[V(F)\cup\{v\}]$ contains a theta-graph and thus contains an even cycle which shares exactly one common vertex $v$ with the odd cycle $D$.
		By Lemma~\ref{ two vertex-disjoint cycles: odd even u}, $G$ contains two cycles of consecutive even lengths.
		So we assume that every vertex in $D$ has at most two neighbors in $F$.
		Now we distinguish among two cases.
		
		\medskip
		
		\textbf{Case I.}  $D$ is a triangle, say $D=v_1v_2v_3v_1$.
		
		\medskip
		
		We note that every leaf of $F$ has at least two neighbors in $D$.
		So it is easy to see that $G$ has a $C_4$. In what follows we prove by showing that $G$ contains a $C_6$.
		
		Consider when $|V(F)|\leq 4$.	
		If $F=K_1$, then all vertices in $G-V(D)$ are isolated and adjacent to all three vertices of $D$ (as $\delta(G)\geq 3$).
		As $|V(G)|\geq 6$, it is easy to see that $G$ has a $C_6$.
		Suppose that $F=K_2$. Let $F=u_1u_2$.
		Then each of $u_1,u_2$ has at least two neighbors in $D$,
		and as $G$ is 3-connected, we have $N_D(u_1)\cup N_D(u_2)=V(D)$.
		Without loss of generality, let us assume that $u_1v_2,u_1v_3,u_2v_1,u_2v_3\in E(G)$.
		Since $|V(G)|\geq 6$, $G-V(D)$ has a second component which is a $K_1$ or $K_2$.
		Let $u_3$ be any vertex in this component, which has at least two neighbors in $D$.
		In any case of $N_D(u_3)$ one can find a $C_6$ in $G[\{u_1,u_2,u_3,v_1,v_2,v_3\}]$ easily.
		If $|V(F)|=3$ (i.e., $F$ is a path $u_1wu_2$),
		then there are two independent edges between $\{u_1,u_2\}$ and $V(D)$, say $u_1v_1,u_2v_2$,
		providing that $u_1wu_2v_2v_3v_1u_1$ is a $C_6$ in $G$.
		Now suppose that $|V(F)|=4$.
		Let $u_1,u_2$ be two leaves of $F$ and $P$ be the path of $F$ between $u_1$ and $u_2$.
		Then $\ell(P)= 2$ or 3.
		There are two independent edges between $\{u_1,u_2\}$ and $V(D)$, say $u_1v_1,u_2v_2$,
		implying that $u_1Pu_2v_2v_3v_1u_1$ (if $\ell(P)=2$) or $u_1Pu_2v_2v_1u_1$ (if $\ell(P)=3$) is a $C_6$ in $G$.
		
		Finally suppose that $F$ has order at least 5.
		Then there are at least $3|V(F)|- 2(|V(F)|-1)\ge 7$ edges between $D$ and $F$, implying that one vertex in $D$ has at least three neighbors in $F$, a contradiction.
		
		\medskip
		
		\textbf{Case II.}  $D$ has length at least 5.
		
		\medskip
		
		Let $D=v_1v_2\ldots v_kv_1$, where $k\geq 5$ is odd.
		If $F$ consists of one vertex $u$, then $u$ has at least three neighbors in $D$ and there exists a shorter odd cycle than $D$ in $G[V(D)\cup \{u\}]$, a contradiction.
		So $F$ is a tree with at least two vertices.
		Then $F$ has at least two leaves, each of which has at least two neighbors in $D$.
		Let $u$ be an arbitrary leaf of $F$.
		We claim that $u$ has exactly two neighbors in $D$, and its two neighbors has distance 2 in $D$.
		Suppose to the contrary that there exists two neighbors $v_1,v_2$ of $u$ has distance in $D$ not equal to 2.
		If $d_D(v_1,v_2)=1$, then $uv_1v_2u$ is a triangle, contradicting the choice of $D$.
		Now assume that $d_D(v_1,v_2)\geq 3$.
		Then both cycles $\overrightarrow{D}[v_1,v_2]\cup \{uv_1, uv_2\}$ and $\overrightarrow{D}[v_2,v_1]\cup \{uv_1, uv_2\}$ have length less than $\ell(D)$,
		but one of them is odd, a contradiction. This proves the claim.
		
		Without loss of generality we assume that $N_D(u)=\{v_1,v_3\}$.
		Note that $G$ has a 4-cycle $v_1v_2v_3uv_1$.
		We further claim that there are no edges between $V(D)\backslash\{v_1,v_2,v_3\}$ and $F$.
		Suppose not. Let $v_iu_1$ be such an edge, where $i\in [4,k]$ and $u_1\in V(F)\backslash\{u\}$.
		Let $P$ be the unique path of $F$ between $u$ and $u_1$.
		Let $C_1=uv_3v_4v_5\ldots v_iu_1\cup P$ and $C_2=uv_1v_kv_{k-1}\ldots v_iu_1\cup P$.
		Then $$\ell(C_1)+\ell(C_2)=\ell(C)+2+2\ell(P)$$ is odd.
		This implies that one of $C_i$'s is even.
		Without loss of generality we assume that $C_1$ is even.
		Let $C'_1=uv_1v_2v_3\ldots v_iu_1\cup P$.
		Then $C_1,C'_1$ are two cycles of consecutive even lengths.
		
		Let $u'$ be a second leaf of $F$.
		Recall that $u'$ has two neighbors in $D$ with distance 2.
		By the previous claim, we have $N_D(u')=\{v_1,v_3\}$.
		Since every vertex in $D$ has at most two neighbors in $F$, $F$ has exactly two leaves, i.e., $F$ is a path.
		Moreover, for every vertex $w\in V(F)\backslash\{u,u'\}$, $w$ has at least one neighbors in $D$ and thus $N_D(w)=\{v_2\}$.
		If $|V(F)|\geq 4$, then $v_2$ is adjacent to every vertex in $V(F)\backslash\{u,u'\}$ and thus $G$ has a triangle,
		contradicting the choice of $D$.
		If $F=K_2$, then $uv_1v_2v_3u'u$ is a $C_5$, implying that $\ell(D)=k=5$. Thus $uv_1v_5v_4v_3u'u$ is a $C_6$.
		Finally if $F$ is a path with three vertices, then $uv_1v_2v_3u'\cup F$ is a $C_6$.
		Recall that $G$ has a $C_4$.
		So $G$ has two cycles of consecutive even lengths.
		This finishes the proof of Theorem~\ref{3-connected}.
	\end{proof}
	
	\section{Concluding remarks}\label{sec:conclusion}
	\subsection{A generalization of Theorem~\ref{admissible: k=3}}
	Following \cite{GHLM}, we say $k$ paths $P_1,P_2,\ldots,P_k$ are {\it admissible},
	if $\ell(P_1)\geq 2$ and $\ell(P_1), \ldots,\ell(P_k)$ form an arithmetic progression of length $k$ with common difference one or two.
	The following is key in \cite{GHLM}.
	
	\begin{thm}[Gao, Huo, Liu and Ma, \cite{GHLM}; see Theorem 3.1]\label{admissible}
		Let $G$ be a graph with $x, y\in V(G)$ such that $G+xy$ is 2-connected.
		If every vertex of $G$ other than $x$ and $y$ has degree at least $k + 1$, then there exist $k$ admissible paths from $x$ to $y$ in $G$.
	\end{thm}	
	
	We can prove a generalization of Theorem~\ref{admissible: k=3} as follows, which may be useful for the general case of Conjecture~\ref{Conj-V}.
	
	\begin{thm}\label{admissible: even}\label{Thm:gen}
		Let $G$ be a graph with $x, y\in V(G)$ such that $G+xy$ is 2-connected. If every vertex of $G$
		other than $x$ and $y$ has degree at least $k + 1$, and any edge $uv\in E(G)$ with $\{u,v\}\cap \{x,y\} = \emptyset$ has $d_G(u)+d_G(v)\geq 2k+3$,
		then $G-xy$ contains either $k+1$ paths from $x$ to $y$ of consecutive lengths or $k$ paths from $x$ to $y$ whose lengths form an arithmetic progression with common difference two.
	\end{thm}
	
	This can be proved by following the original proof of Theorem~\ref{admissible} in \cite{GHLM} with some necessary modifications.
	Here we give a sketch.
	We prove by induction on $V(G)+E(G)$ and $k$.
	The base case is just Theorem~\ref{admissible: k=3}.
	For the inductive proof, the main idea is to consider a new graph $G'$ obtained from $G$ by contracting some special set $W$ of vertices
	such that $d_{G'}(v)\ge d_G(v)-t$ holds for some integer $t>0$ and vertices $v\notin \{x,y\}$; then we can apply induction hypothesis for $k-t$ to $G'$.
	To get the set $W$, in the proof of Theorem~\ref{admissible} we need to find some local structure containing $x$, which is either a clique (see Lemma~3.6) or a complete bipartite graph (see Lemma~3.7).
	In the case that a complete bipartite graph contains $x$, one can always concatenate consecutive even paths and admissible paths, so this part of proof (after Lemma~3.7) does not need any modification.
	So let us assume that there exists a clique $K$ of size $k+1$ containing $x$.
	Here we need more extra discussions.
	In this case, since any edge $uv$ with $\{u,v\}\cap \{x,y\} = \emptyset$ satisfies $d_G(u)+d_G(v)\ge 2k+3$,
	we have $V(G)\setminus(K\cup \{y\})\ne \emptyset$.
	If there is no vertex $u$ with at least two neighbors in $K\setminus\{x\}$, then we contract $K\setminus\{x\}$ and the proof is similar to that of Theorem~\ref{admissible};
	otherwise, let $K'=K\cup \{u\}$ and $P$ be a path from $y$ to $K'$ not through $x$, then we can get $k+1$ consecutive paths from $x$ to $y$ in $G$.
	
	\subsection{Related open problems}
	Our main result provides the tight average degree condition for forcing two cycles of consecutive even lengths and thus settles the case $k=2$ of Conjecture~\ref{Conj-V}.
	It would be very intersecting and natural to pursue the remaining cases (for $k\geq 3$) of the conjecture.
	In fact there also is a slightly stronger conjecture made by Sudakov and Verstra\"ete \cite{SV17}.
	\begin{conj}[Sudakov and Verstra\"ete \cite{SV17}; see Conjecture~8]\label{Conj-SV}
		Let $k\ge 1$. If $G$ is a graph with a maximum number of edges that does not contain cycles of $k$ consecutive even lengths, then every block of $G$ is a complete graph of order at most $2k + 1$.
	\end{conj}
	
	Recall the definition of $c_{\ell,k}$, that is the least constant such that every $n$-vertex graph with at least $c_{\ell,k}\cdot n$ edges contains a cycle of length $\ell$ modulo $k$.
	Erd\H{o}s \cite{Erd95} asked for the exact value of $c_{\ell,k}$ for all integers $k>\ell\geq 0$.
	For any $k>\ell\geq 3$, Sudakov and Verstra\"ete \cite{SV17} determined the value of $c_{\ell,k}$ up to a constant factor,
	by showing that $d_{\ell,k}\leq c_{\ell,k}\leq 96\cdot d_{\ell,k}$, where $d_{\ell,k}$ denotes the largest possible average degree of any $k$-vertex $C_\ell$-free graph.
	For the case when $\ell\geq 4$ is even,
	the above result of Sudakov-Verstra\"ete shows that determining the order of the magnitude of $c_{\ell,k}$ is as hard as the famous extremal problem of determining the Tur\'an number of the even cycle $C_\ell$;
	and in this case, Sudakov and Verstra\"ete \cite{SV17} further conjectured that $\lim_{k\to \infty} c_{\ell,k}/d_{\ell,k}=1$.
	For the case when $\ell=2$, we like to make the following conjecture.
	\begin{conj}
		For any integer $k\geq 2$, it holds that $c_{2,2k}=(2k+1)/2$.
	\end{conj}
	
	Note that the case $k=2$ is given by Corollary~\ref{cor}, and Conjecture~\ref{Conj-V} (or Conjecture~\ref{Conj-SV}) would imply the above conjecture.
	For the case when $\ell=0$, it seems hard to make a reasonable prediction.
	The best bounds for $c_{0,4}$ that we are aware only give that $11/7\le c_{0,4}\leq 2$ (see \cite{DLS93} for the upper bound).
	For other related interesting problems, we direct readers to \cite{SV17,V16}.

\end{document}